\documentclass[11pt,a4paper]{amsart}
\usepackage[margin=28mm]{geometry}
\usepackage{amsmath,amssymb,amsthm}
\usepackage{color}
\usepackage[unicode,colorlinks=true,linkcolor=blue,citecolor=blue,urlcolor=blue]{hyperref}
\newtheorem{theorem}{Theorem}[section]
\newtheorem{lemma}[theorem]{Lemma}
\newtheorem{corollary}[theorem]{Corollary}
\newtheorem{proposition}[theorem]{Proposition}
\theoremstyle{remark}
\newtheorem{remark}[theorem]{Remark}
\newcommand{\KK}{\mathbb{K}}
\newcommand{\g}{\mathfrak{g}}
\newcommand{\h}{\mathfrak{h}}
\DeclareMathOperator{\Lie}{Lie}
\DeclareMathOperator{\Aut}{Aut}
\DeclareMathOperator{\End}{End}
\DeclareMathOperator{\Ad}{Ad}
\DeclareMathOperator{\ad}{ad}
\DeclareMathOperator{\GL}{GL}
\DeclareMathOperator{\dl}{dl}
\newcommand{\lalg}{\mathfrak l}
\theoremstyle{definition}
\newtheorem{definition}[theorem]{Definition}
\title[Solvability and nilpotence of ind-groups]{Solvability 
and nilpotence of ind-groups via tangent curves}
\author{Alexander Perepechko}
\address{HSE University, Faculty of Computer Science,
Pokrovsky blvd. 11, Moscow, 109028 Russia}
\email{a@perep.ru}
\thanks{The research of A.~Perepechko was supported by RSF grant
No.~25-11-00302.}
\author{Mikhail Zaidenberg}
\address{Institut Fourier, UMR 5582, Laboratoire de Math\'ematiques,
Universit\'e Grenoble Alpes, CS 40700, 38058 Grenoble cedex 9, France}
\email{mikhail.zaidenberg@gmail.com}
\subjclass[2020]{Primary 22E65; Secondary 14L40, 17B30, 17B66}
\keywords{ind-group, automorphism group of an affine variety,
Lie algebra, solvable group, nilpotent group, logarithmic derivative}
\date{}
\begin{document}
\begin{abstract}
Answering a question of Kraft and the second-named author~\cite{KZ},
we prove that a connected ind-group over an algebraically closed
field of characteristic zero is solvable, respectively nilpotent, if and
only if its Lie algebra is. Their derived lengths, respectively
nilpotency classes, coincide.
We also obtain bounds on the derived length and nilpotency class
for regular linear
representations of connected ind-groups and for subgroups generated by
connected algebraic subgroups. As an application, we give a new proof
of the geometric case of a theorem of Belabbas~\cite{Belabbas}:
a solvable Lie algebra
of vector fields on an affine variety generated by finitely many locally
finite Lie subalgebras is locally finite and finite dimensional.
\end{abstract}
\maketitle

\section{Introduction}

Throughout, $\KK$ is an algebraically closed field of characteristic zero.
An ind-group $G$ has an admissible filtration
\begin{equation}\label{eq:filtration}
 G_0\subseteq G_1\subseteq\cdots,\qquad G=\bigcup_{N\geq0}G_N,
\end{equation}
where every $G_N$ is an algebraic variety and every inclusion is a closed
immersion. 
An ind-variety $Y$ is called \emph{curve-connected} if, for any
$a,b\in Y$, there are an irreducible algebraic curve $C$
and a morphism $\alpha:C\to Y$ such that $a,b\in\alpha(C)$;
see \cite[Definition~1.6.1]{FK}.
For ind-groups, connectedness is equivalent to curve-connectedness
\cite[Proposition 2.2.1 and Remark 2.2.3]{FK}.

For a solvable group $H$
or Lie algebra $\lalg$, we write $\dl(H)$ or $\dl(\lalg)$ for its derived
length, the least $d\geq0$ for which the $d$th term of its derived series
is trivial.
For a nilpotent group $H$ or Lie algebra $\lalg$, we write $c(H)$ or
$c(\lalg)$ for its nilpotency class, the least $r\geq0$ for which the
$(r+1)$st term of its lower central series is trivial. 

\begin{theorem}\label{thm:main}
Let $G$ be a connected ind-group and $\g:=\Lie G$.
Then $G$ is solvable, respectively nilpotent, if and only if $\g$ is.
Moreover, $\dl(G)=\dl(\g)$ whenever either is solvable, and
$c(G)=c(\g)$ whenever either is nilpotent.
\end{theorem}

\begin{corollary}\label{cor:linear}
Let $G$ be a connected ind-group, and let
$\rho:G\to\operatorname{GL}(V)$ be a regular linear representation on a
countable-dimensional vector space $V$.
Let $\h:=d\rho_e(\Lie G)$. If $\h$ is solvable (respectively, nilpotent),
then so is $\rho(G)$, with
$\dl(\rho(G))\leq\dl(\h)$ (respectively, $c(\rho(G))\leq c(\h)$).
\end{corollary}

Corollary~\ref{cor:linear} is not an immediate consequence of
Theorem~\ref{thm:main}; see Remark~\ref{rem:linear}.

Theorem~\ref{thm:main} answers Question~5 in
\cite[Section~5.1]{KZ} affirmatively for connected ind-groups,
with equality of derived lengths as well as nilpotency classes.
In particular, if $X$ is a variety and $G\subseteq\Aut(X)$ a subgroup 
equipped with the
structure of a connected ind-group, then Theorem \ref{thm:main} applies.
 If $G$ is a connected ind-group and $\rho: G \to \Aut(X)$ 
 comes from a regular $G$-action on $X$, 
 then Corollary \ref{cor:linear} applies.
The connectedness assumption is essential: the finite subgroup
$A_5\subset\mathrm{GL}_5(\KK)\subset\Aut(\mathbb A^5)$, acting by
permutation matrices, has zero Lie algebra and is not solvable.
For commutativity, the equivalence was already known
\cite[Corollary~7.5.3]{FK}; see also \cite[Remark~5.1.5]{KZ}.

A further consequence concerns a subgroup 
$H$ generated by a family of
connected algebraic subgroups of an ind-group.

\begin{corollary}\label{cor:algebraic}
Let $G$ be an ind-group and let
$\{H_i\}_{i\in I}$ be connected algebraic subgroups of $G$. Set
\[
 H:=\langle H_i\mid i\in I\rangle,\qquad
 \h:=\langle\Lie H_i\mid i\in I\rangle_{\mathrm{Lie}}.
\]
Then $H$ is solvable (respectively, nilpotent) if and only if $\h$ is.
Moreover, $\dl(H)=\dl(\h)$ whenever either is solvable, and
$c(H)=c(\h)$ whenever either is nilpotent.
\end{corollary}

Corollary~\ref{cor:algebraic} answers affirmatively the solvability question
raised in \cite[Remark~2.10]{ZLF}.
When the family is finite and the ambient group is
$\Aut(X)$ for a  quasi-affine variety $X$, solvability makes $H$ algebraic by
\cite[Theorem~A]{CKRVS}, with $\h=\Lie H$ by \cite[Theorem~A]{KZ} 
in the affine case. 

A Lie algebra $\lalg$ of derivations of a $\KK$-algebra $A$ is called
\emph{locally finite} if $A$ is a union of finite-dimensional
$\lalg$-invariant subspaces.
The following corollary gives an alternative proof of the geometric case of
Theorem~D in \cite{Belabbas}, which answers Question~1 in \cite{AZ25}
affirmatively.

\begin{corollary}\label{cor:application-locally-finite}
Let $X$ be an affine variety and let
$\mathfrak h\subseteq\Lie\Aut(X)$ be a solvable Lie algebra generated by
locally finite Lie subalgebras $\h_i$, $i=1,\ldots,n$.  
Then $\h$ is locally
finite, and in particular finite dimensional. If $\h_i=\Lie H_i$, 
where the $H_i\subset\Aut(X)$ are connected 
algebraic subgroups, then 
$\h=\Lie H$, where $H=\langle H_i\mid i=1,\ldots,n\rangle$.
\end{corollary}

The proof of Theorem~\ref{thm:main} uses logarithmic derivatives
along curves in $G$ through the identity.
This method was initiated in \cite[Section~3]{P}, where formal
expansions were used. Here we develop this approach by following
curves tangent to Lie ideals. To each $\Ad(G)$-stable ideal
$\mathfrak a\subseteq\mathfrak g$, we associate the normal subgroup
$G(\mathfrak a)\triangleleft G$ generated by the images of morphisms
from smooth irreducible curves through the identity whose logarithmic
derivatives take values in $\mathfrak a$.
Connectedness gives $G(\mathfrak g)=G$, while $G(0)=\{e\}$.
The key commutator lemma (Lemma~\ref{lem:commutator}) states that
\begin{equation}\label{eq:comm}
 [G(\mathfrak a),G(\mathfrak b)]
 \subseteq G([\mathfrak a,\mathfrak b])
\end{equation}
for any two $\Ad(G)$-stable ideals $\mathfrak a,\mathfrak b
\subseteq\mathfrak g$.
Induction then transfers the derived-length and nilpotency-class
bounds from $\mathfrak g$ to $G$. For the reverse bounds,
we adapt the proof of \cite[Lemma~5.1.4]{KZ} from the affine case
to our more general setting.

For a regular representation, curves tangent to the kernel of its
differential act trivially, giving Corollary~\ref{cor:linear}.
We derive the differentiation formulas in Section~\ref{sec:preliminaries}
and prove the commutator inclusion \eqref{eq:comm} 
in Section~\ref{sec:commutator}.
The proofs of the main results and the applications follow in
Sections~\ref{sec:main-proofs} and~\ref{sec:algebraically-generated}.

\subsection*{Declaration of AI assistance}
The authors used OpenAI models to draft proofs and improve the language
and exposition. The authors are responsible for the manuscript's
correctness and originality.

\section{Preliminaries}\label{sec:preliminaries}

\subsection{Ind-groups and their representations}

For the definitions of ind-varieties, ind-morphisms, and ind-groups,
we follow \cite[Sections~1.1 and~2.1]{FK}.

Let $G:=\bigcup_NG_N$ be an ind-group, with $e$ in
every filtration element. Its Lie algebra is the countable-dimensional
space $\Lie G=\varinjlim_N T_eG_N$, with its usual
Lie bracket. We use the following facts.
\begin{enumerate}
\item A connected ind-group is curve-connected
\cite[Proposition~2.2.1(2)]{FK}.
Normalizing a connecting curve gives a morphism from a smooth
irreducible curve joining any prescribed two points.
\item Conjugation induces a regular linear action of $G$ on
$\Lie G$, and
\[
 d\Ad_e(x)(y)=\ad(x)(y)=[x,y]\qquad(x,y\in\Lie G).
\]
For Lie groups, the identity $d\Ad_e=\ad$ is proved in
\cite[Chapter~I, \S2, Theorem~3]{OV}.
For affine ind-groups, regularity and this identity are proved in
\cite[Theorem~7.4.4 and Proposition~7.5.2]{FK}.
The same statements hold without affineness by
Lemmas~\ref{lem:adjoint-regular} and~\ref{lem:adjoint-differential} below.
\item Every regular linear representation $\rho:G\to\GL(V)$
on a countable-dimensional vector space $V$ induces a natural Lie algebra
homomorphism $d\rho_e:\Lie G\to\End_\KK(V)$
\cite[Definition~2.6.1 and Lemma~2.6.2]{FK};
for the classical construction, see \cite[Chapter~I, \S2, subsection~2]{OV}.
Here regularity means that the action map
$G\times V\to V$, $(g,v)\mapsto\rho(g)v$, is an ind-morphism,
where $V$ has the ind-variety structure given by its finite-dimensional
subspaces.
For $x\in\Lie G=T_eG$ and $v\in V$, the orbit map
$\mu_v:G\to V$, $g\mapsto\rho(g)v$, gives
\[
 (d\rho_e(x))(v)=(d\mu_v)_e(x).
\]
Every orbit map restricted to an algebraic curve takes values in a
finite-dimensional subspace of $V$.
\end{enumerate}

\begin{lemma}\label{lem:adjoint-regular}
For any ind-group $G$, conjugation induces a regular linear action
of $G$ on $\Lie G$.
\end{lemma}

\begin{proof}
As in the proof of \cite[Theorem~7.4.4]{FK}, we work on filtration
elements. For fixed $i,j$, conjugation factors through
\[
 c:G_i\times G_j\longrightarrow G_N,
 \qquad c(g,h)=ghg^{-1}.
\]
Writing $TZ$ for the tangent bundle of a variety $Z$
(see \cite[AG, 16.1--16.3]{Bor91} for the possibly singular case),
we consider the composition
\[
 G_i\times T_eG_j
 \xrightarrow{\ (g,y)\mapsto(0_g,y)\ }
 T(G_i\times G_j)
 \xrightarrow{Tc}TG_N,
\]
where $0_g$ is the zero tangent vector at $g$.
Since $c(g,e)=e$, this is a regular map into $T_eG_N$,
sending $(g,y)$ to $\Ad(g)y$.
These maps are compatible as $i,j$ increase, 
so they define an ind-morphism
$G\times\Lie G\to\Lie G$.
Finally, differentiating $c(gh,u)=c(g,c(h,u))$ 
with respect to $u$ at $u=e$ gives
$\Ad(gh)=\Ad(g)\Ad(h)$.
\end{proof}

\begin{lemma}\label{lem:adjoint-differential}
For any ind-group $G$ and $x,y\in\Lie G$, one has
\[
 (d\Ad)_e(x)(y)=[x,y].
\]
\end{lemma}

\begin{proof}
Let us choose a filtration member $G_i$ such that $x,y\in T_eG_i$.

Let $D_\varepsilon:=\KK[\varepsilon]/(\varepsilon^2)$.
The tangent vector $x$ defines an infinitesimal point
$a\in G_i(D_\varepsilon)$ with underlying point $e$, denoted by
$a=e+\varepsilon x$. This notation means that, for every
$F\in\mathcal O_{G_i,e}$,
\begin{equation}\label{eq:infinitesimal-point}
 F(a)=F(e)+\varepsilon\,dF_e(x).
\end{equation}
We define $b=e+\eta y$ in the same way. 
To use both points together,
we extend scalars to
$R=\KK[\varepsilon,\eta]/(\varepsilon^2,\eta^2)$.

We write $\mu^a(g):=ga$ for 
the infinitesimal right translation by $a$.
Its pullback $T_a:=(\mu^a)^*$ acts on local functions by
\[
 (T_af)(g)=f(ga).
\]

We write $\mu_g(h):=gh$ for the left translation.
Let $v_x$ be the left-invariant field
whose value at $g$ is $(d\mu_g)_e(x)$. 
It acts on local functions by
\[
 (v_xf)(g):=df_g(v_x(g))
 =df_g\bigl((d\mu_g)_e(x)\bigr)=d(f\circ\mu_g)_e(x).
\]
Applying~\eqref{eq:infinitesimal-point} 
to $h\mapsto f(gh)$ gives
\[
 (T_af)(g)
 =f(g)+\varepsilon\,d(f\circ\mu_g)_e(x)
 =f(g)+\varepsilon\,(v_xf)(g).
\]
Thus $T_a=1+\varepsilon v_x$.
Similarly, $T_b=1+\eta v_y$.
We have
\[
 (T_aT_bf)(g)=(T_bf)(ga)=f(gab),
\]
so $T_aT_b=T_{ab}$.
Since $\varepsilon^2=\eta^2=0$, 
the inverse translations give
$T_{a^{-1}}=1-\varepsilon v_x$ and
$T_{b^{-1}}=1-\eta v_y$. Consequently,
\[
 T_{aba^{-1}b^{-1}}
 =(1+\varepsilon v_x)(1+\eta v_y)
  (1-\varepsilon v_x)(1-\eta v_y)
 =1+\varepsilon\eta[v_x,v_y].
\]
Here $[v_x,v_y]=v_xv_y-v_yv_x$, and its value at $e$ is
$[x,y]$ by the definition of the Lie bracket on $\Lie G$.
For every local function $f$ near $e$, 
applying this operator identity
and evaluating at $e$ gives
\begin{align*}
 f(aba^{-1}b^{-1})
 &=(T_{aba^{-1}b^{-1}}f)(e)\\
 &=f(e)+\varepsilon\eta\,([v_x,v_y]f)(e)\\
 &=f(e)+\varepsilon\eta\,df_e([x,y]).
\end{align*}
Since this holds for every $f$, the description 
of infinitesimal points
in~\eqref{eq:infinitesimal-point}, 
with $\varepsilon\eta$ in place of
$\varepsilon$, gives
\[
 aba^{-1}b^{-1}=e+\varepsilon\eta[x,y].
\]
On the other hand, since $a=e+\varepsilon x$, applying
\eqref{eq:infinitesimal-point} to the orbit map 
$g\mapsto\Ad(g)y$ gives
\[
 \Ad(a)y=y+\varepsilon(d\Ad)_e(x)(y).
\]
Since $b=e+\eta y$, conjugation by $a$ then gives
\[
 aba^{-1}=e+\eta\Ad(a)y
 =e+\eta y+\varepsilon\eta(d\Ad)_e(x)(y).
\]
Multiplication by $b^{-1}=e-\eta y$ cancels the term $\eta y$.
Comparing with the
preceding expression for the commutator yields
$(d\Ad)_e(x)(y)=[x,y]$.
\end{proof}

\subsection{Logarithmic derivatives}

We write $\mu_g(h):=gh$ and $\mu^g(h):=hg$ 
for left and right translations in $G$.

\begin{definition}
\label{def:curve-differentials}
Let $C$ be a smooth irreducible algebraic curve and
$\alpha:C\to G$ a morphism. For $t\in C$ and $v\in T_tC$,
we define its \emph{right logarithmic derivative} by
\begin{equation}\label{eq:curve-log}
 (\delta^r\alpha)_t(v)
 :=d(\mu^{\alpha(t)^{-1}})_{\alpha(t)}(d\alpha_t(v))
 \in\Lie G.
\end{equation}
Thus, $\delta^r\alpha$ transports the velocities of $\alpha$ to the single
fixed vector space $\Lie G=T_eG$, rather than leaving them in the varying
tangent spaces $T_{\alpha(t)}G$.
In Maurer--Cartan terminology, $\delta^r\alpha=\alpha^*\omega^r$ is the
pullback along $\alpha$ of the right Maurer--Cartan form $\omega^r$;
see Remark~\ref{rem:maurer-cartan}.
For a subspace $\mathfrak a\subseteq\Lie G$, saying that
$\delta^r\alpha$ takes values in $\mathfrak a$ means
$(\delta^r\alpha)_t(T_tC)\subseteq\mathfrak a$ at every point $t$.
For a matrix-valued curve and a local parameter $z$,
\[
 \delta^r\alpha=\alpha'(z)\alpha(z)^{-1}\,dz.
\]
For Lie groups, the value of this form on $\partial/\partial z$ is
precisely the velocity $\xi(z)$ defined in
\cite[Chapter~I, \S2, subsection~5, pp.~25--26]{OV}.
\end{definition}

\begin{remark}\label{rem:maurer-cartan}
For a Lie group $G$, the right Maurer--Cartan form is
\[
 \omega^r_g=d(\mu^{g^{-1}})_g:T_gG\longrightarrow\Lie G.
\]
It identifies each tangent space with $T_eG$, giving a trivialization
$TG\simeq G\times\Lie G$. The right logarithmic derivative is its pullback:
\[
 \delta^r\alpha=\alpha^*\omega^r.
\]
The form $\omega^r$ is right invariant and satisfies the Maurer--Cartan
equation
\[
 d\omega^r-\tfrac12[\omega^r,\omega^r]=0.
\]
Our construction applies this tangent-space identification along
algebraic curves within an ind-group.
\end{remark}

\begin{remark}\label{rem:curve-log-regular}
The differential in~\eqref{eq:curve-log} is regular and takes values in
a fixed finite-dimensional subspace of $\Lie G$.
Indeed, the morphism
\[
 F:C\times C\longrightarrow G,\qquad
 F(s,t)=\alpha(s)\alpha(t)^{-1}
\]
factors through a filtration element $G_N$, and
\[
 (\delta^r\alpha)_t(v)=dF_{(t,t)}(v,0)\in T_eG_N.
\]
Let us choose an affine neighbourhood $U$ of $e$ in $G_N$ and a
closed immersion $U\hookrightarrow\mathbb A^m$.
Differentiating $s\mapsto F(s,t)$ at $s=t$, for each $t\in C$,
gives a regular bundle map $TC\to C\times\KK^m$.
Since $F$ maps a neighbourhood of the diagonal into $U$ and $F(t,t)=e$,
these tangent vectors lie in $T_eU=T_eG_N$.
\end{remark}

\begin{lemma}
\label{lem:differentiation}
For morphisms $\alpha,\beta:C\to G$ and a fixed
$x\in\Lie G$, one has the following equalities 
of $\Lie G$-valued $1$-forms on $C$:
\begin{align}
 \delta^r(\alpha\beta)
 &=\delta^r\alpha+\Ad(\alpha)\delta^r\beta,\label{eq:product}\\
 \delta^r(\alpha^{-1})
 &=-\Ad(\alpha)^{-1}\delta^r\alpha,\label{eq:inverse}\\
 d\bigl(\Ad(\alpha)x\bigr)
 &=[\delta^r\alpha,\Ad(\alpha)x].\label{eq:adjoint}
\end{align}
The bracket in~\eqref{eq:adjoint} is evaluated on tangent vectors.
\end{lemma}

\begin{proof}
For fixed $t\in C$ and a tangent vector $\xi\in T_tC$, we set
$a:=\alpha(t)$, $b:=\beta(t)$, and
\[
 u_\alpha:=d\alpha_t(\xi)\in T_a G,
 \qquad
 u_\beta:=d\beta_t(\xi)\in T_b G.
\]
Thus $u_\alpha$ and $u_\beta$ are the
velocities of $\alpha$ and $\beta$ in the same direction $\xi\in T_tC$.
If $z$ is a local parameter and $\xi=\partial/\partial z$, then simply
$u_\alpha=\alpha'(z)$ and $u_\beta=\beta'(z)$. 

Let us turn to the product formula \eqref{eq:product}.
Let $\mu:G\times G\to G$ be the multiplication map. Its differential at
$(a,b)$ is
\[
 d\mu_{(a,b)}(u_\alpha,u_\beta)
 =d(\mu^b)_a(u_\alpha)+d(\mu_a)_b(u_\beta).
\]
Hence
\[
 d(\alpha\beta)_t(\xi)
 =d(\mu^b)_a(u_\alpha)+d(\mu_a)_b(u_\beta).
\]
To obtain the right logarithmic derivative, we apply
$d(\mu^{(ab)^{-1}})_{ab}$.  For the first summand, we get
\[
 \mu^{(ab)^{-1}}\circ \mu^b=\mu^{a^{-1}}.
\]
It follows that
\[
 d(\mu^{(ab)^{-1}})_{ab}\,d(\mu^b)_a(u_\alpha)
 =d(\mu^{a^{-1}})_a(u_\alpha)
 =(\delta^r\alpha)_t(\xi).
\]
For the second summand, let $c_a(h):=aha^{-1}$.  Then
\[
 \mu^{(ab)^{-1}}\circ \mu_a=c_a\circ \mu^{b^{-1}}.
\]
The differential of $c_a$ at $e$ is $\Ad(a)$.  Therefore
\[
 d(\mu^{(ab)^{-1}})_{ab}\,d(\mu_a)_b(u_\beta)
 =\Ad(a)\,d(\mu^{b^{-1}})_b(u_\beta)
=\Ad(a)(\delta^r\beta)_t(\xi).
\]
Adding the two terms gives
\[
 \bigl(\delta^r(\alpha\beta)\bigr)_t(\xi)
 =(\delta^r\alpha)_t(\xi)
 +\Ad(\alpha(t))(\delta^r\beta)_t(\xi).
\]
Since $t$ and $\xi$ were arbitrary,  
\eqref{eq:product} follows. 

To prove the inversion formula \eqref{eq:inverse}, we
apply the product formula to the constant curve
$\alpha\alpha^{-1}=e$.  
The logarithmic derivative of a constant curve is
zero, so
\[
 0=\delta^r(\alpha\alpha^{-1})
 =\delta^r\alpha+\Ad(\alpha)\delta^r(\alpha^{-1}).
\]
Applying $\Ad(\alpha)^{-1}$ gives \eqref{eq:inverse}.

To prove formula \eqref{eq:adjoint} 
for the differential of the adjoint action,
let us fix an arbitrary $x\in \Lie G$ and consider
\[
 f:C\longrightarrow\Lie G,
 \qquad
 f(t)=\Ad(\alpha(t))x.
\]
For fixed $t$ and $\xi\in T_tC$, we set
$a:=\alpha(t)$ and $y:=(\delta^r\alpha)_t(\xi)$.
The definition of the right logarithmic derivative says
\[
 y=d(\mu^{a^{-1}})_a(d\alpha_t(\xi)).
\]
Since $\mu^a$ and $\mu^{a^{-1}}$ 
are inverse maps, this is equivalent to
\[
 d\alpha_t(\xi)=d(\mu^a)_e(y).
\]
Thus the velocity of $\alpha$ at $a$ is obtained 
by right-translating $y\in T_e G$ to $a$.

To differentiate $g\mapsto\Ad(g)x$ at $g=a$ 
in this direction, we use the map $h\mapsto\Ad(ha)x$.
The map $h\mapsto ha$ sends $e$ to $a$, 
and its differential sends $y$ to $d\alpha_t(\xi)$.  Moreover,
\[
 \Ad(ha)x=\Ad(h)\Ad(a)x.
\]
Consequently,
\[
 df_t(\xi)=(d\Ad)_e(y)\bigl(\Ad(a)x\bigr).
\]
Using $(d\Ad)_e(y)(z)=[y,z]$, we obtain
\[
 df_t(\xi)=[y,\Ad(a)x]
 =[\,(\delta^r\alpha)_t(\xi),\Ad(\alpha(t))x\,].
\]
Now  \eqref{eq:adjoint} follows.
\end{proof}
\begin{remark}
For matrices, the product and adjoint formulas 
\eqref{eq:product} and \eqref{eq:adjoint} reduce to
\[
 (\alpha\beta)'(\alpha\beta)^{-1}
 =\alpha'\alpha^{-1}
  +\alpha(\beta'\beta^{-1})\alpha^{-1}
\]
and
\[
 (\alpha x\alpha^{-1})'
 =[\alpha'\alpha^{-1},\alpha x\alpha^{-1}].
\]
\end{remark}

\subsection{Vanishing differentials on curves}

\begin{lemma}\label{lem:constant}
Let $C$ be a smooth irreducible algebraic curve.
A morphism $\varphi:C\to Y$ to an ind-variety
$Y$ with zero differential is constant.
Moreover, if $f:C\to V$ is a regular map to a countable-dimensional
vector space and $W\subseteq V$ is a linear subspace, then
$df(TC)\subseteq W$ implies that $f\bmod W$ is constant.
\end{lemma}

\begin{proof}
The morphism $\varphi$ factors through a filtration member
$Y_n$. Clearly, its differential in $Y_n$ is also zero.
A nonconstant morphism from a curve in
characteristic zero has nonzero differential at a general point,
which proves the first assertion.

For the second, we choose a finite-dimensional subspace $V_0$ containing
$f(C)$ and apply the first assertion to
$C\to V_0/(V_0\cap W)$.
\end{proof}

\section{The commutator lemma}\label{sec:commutator}

In this section, we use curves tangent to Lie ideals to transfer solvability
and nilpotency bounds from Lie algebras to groups.
Throughout, we fix a subgroup
$H\subseteq G$ and an $\Ad(H)$-stable Lie subalgebra
$\lalg\subseteq\Lie G$.
We regard $H$ and the subgroups constructed below as abstract groups.

We use $[a,b]:=aba^{-1}b^{-1}$ and denote by $[A,B]$ the subgroup
generated by the commutators $[a,b]$ with $a\in A$ and $b\in B$.

\begin{definition}\label{def:tangent-subgroup}
Let $\mathfrak a\subseteq\lalg$ be an $\Ad(H)$-stable Lie ideal.
We say that a morphism $\alpha:C\to G$ from a smooth irreducible
algebraic curve $C$ is \emph{tangent to $\mathfrak a$} if
\begin{equation}\label{eq:admissible}
 (\delta^r\alpha)_t(T_tC)\subseteq\mathfrak a
 \qquad\text{for every }t\in C.
\end{equation}
We denote by $H(\mathfrak a)$ the subgroup generated by the images of
all curves tangent to $\mathfrak a$ whose images lie in $H$ and contain $e$.
For such a curve, we choose a point $p\in C$ with $\alpha(p)=e$
and call it the \emph{marked point}.
\end{definition}

This construction is analogous to the integral subgroup associated with
a Lie subalgebra in classical Lie theory; see
\cite[Chapter~I, \S2, Theorem~8]{OV}. Here we use algebraic curves in an
ind-group.

The construction is monotone in $\mathfrak a$.
Conjugation by $h\in H$ sends a curve tangent to $\mathfrak a$ to another
such curve, because its translated tangent vectors are obtained by
applying $\Ad(h)$. Consequently,
\begin{equation}\label{eq:normal-zero}
 H(\mathfrak a)\triangleleft H,\qquad H(0)=\{e\}.
\end{equation}
The second assertion follows from Lemma~\ref{lem:constant},
since right translation is an isomorphism on tangent spaces.

\begin{lemma}\label{lem:commutator}
For $\Ad(H)$-stable ideals $\mathfrak a,\mathfrak b$ of $\lalg$, one has
\[
 [H(\mathfrak a),H(\mathfrak b)]
 \subseteq H([\mathfrak a,\mathfrak b]).
\]
\end{lemma}

\begin{proof}
We first prove that $H(\mathfrak b)$ acts trivially on
$\mathfrak a/[\mathfrak a,\mathfrak b]$.
We set $\mathfrak c:=[\mathfrak a,\mathfrak b]$, an $\Ad(H)$-stable ideal
of $\lalg$. Let $\alpha:C\to G$ be one of the curves generating $H(\mathfrak b)$,
and let $x\in\mathfrak a$.
Formula~\eqref{eq:adjoint} gives
\[
 d\bigl(\Ad(\alpha)x\bigr)
 =[\delta^r\alpha,\Ad(\alpha)x].
\]
Since $\delta^r\alpha$ takes values in $\mathfrak b$ and
$\Ad(\alpha)x\in\mathfrak a$, this differential takes values in
$\mathfrak c$. By Lemma~\ref{lem:constant},
$\Ad(\alpha(t))x\bmod\mathfrak c$ is constant, with value
$x\bmod\mathfrak c$ at the marked point. Thus every element in the image
of $\alpha$ acts trivially on $\mathfrak a/\mathfrak c$ via the adjoint
action. Since the images of all such curves generate $H(\mathfrak b)$,
every $g\in H(\mathfrak b)$ acts trivially on this quotient. Equivalently,
\begin{equation}\label{eq:adjoint-difference}
 (\Ad(g)-1)\mathfrak a\subseteq\mathfrak c
 \qquad\text{for any }g\in H(\mathfrak b).
\end{equation}

Let us fix such a $g$ and one of the curves $\gamma:S\to G$ generating
$H(\mathfrak a)$, with marked point $p$. We set
$\eta(s):=[\gamma(s),g]$.
Since $g$ is constant, $\delta^r g=\delta^r(g^{-1})=0$.
Using~\eqref{eq:product} and~\eqref{eq:inverse}, we obtain
\begin{align*}
 \delta^r\eta
 &=\delta^r(\gamma g\gamma^{-1})\\
 &=\delta^r\gamma+\Ad(\gamma g)\delta^r(\gamma^{-1})\\
 &=\delta^r\gamma-\Ad(\gamma g)\Ad(\gamma)^{-1}\delta^r\gamma\\
 &=\delta^r\gamma-\Ad(\gamma g\gamma^{-1})\delta^r\gamma\\
 &=\Ad(\gamma)(1-\Ad(g))\Ad(\gamma)^{-1}\delta^r\gamma.
\end{align*}
The last two equalities use $\Ad(uv)=\Ad(u)\Ad(v)$.
At $\gamma(p)=e$, this reduces to the differential $1-\Ad(g)$ of
$h\mapsto[h,g]$ at $e$; cf. \cite[proof of Lemma~5.1.4]{KZ},
where the opposite order $h\mapsto[g,h]$ gives $\Ad(g)-1$.
Thus $\eta$ is tangent to $\mathfrak c$ by~\eqref{eq:adjoint-difference}
and invariance of $\mathfrak a$ and $\mathfrak c$. Since $\eta(p)=e$ and
$\eta(S)\subseteq H$, we have $\eta(S)\subseteq H(\mathfrak c)$.
By~\eqref{eq:normal-zero}, $H(\mathfrak c)$ is normal in $H$.
In the quotient $H/H(\mathfrak c)$, the image of every generator of
$H(\mathfrak a)$ commutes with the image of $H(\mathfrak b)$.
Hence the images of $H(\mathfrak a)$ and $H(\mathfrak b)$ commute, so
$[H(\mathfrak a),H(\mathfrak b)]\subseteq H(\mathfrak c)$.
\end{proof}

For a group $H$, we write
\[
 H^{(0)}:=H,\qquad H^{(r+1)}:=[H^{(r)},H^{(r)}],\qquad
 H_1:=H,\qquad H_{r+1}:=[H,H_r]
\]
for its derived and lower central series, respectively.
The Lie-algebra series have the same notation and indexing.

\begin{corollary}\label{cor:transfer}
If $H(\lalg)=H$, then
\[
 H^{(r)}\subseteq H(\lalg^{(r)})\quad(r\geq0),\qquad
 H_r\subseteq H(\lalg_r)\quad(r\geq1).
\]
In particular, if $\lalg$ is solvable (respectively, nilpotent), then
so is $H$, with $\dl(H)\leq\dl(\lalg)$ (respectively, $c(H)\leq c(\lalg)$).
\end{corollary}

\begin{proof}
The derived and lower central ideals of $\lalg$ are $\Ad(H)$-stable.
The assertions follow by induction from Lemma~\ref{lem:commutator},
starting with $H(\lalg)=H$ and using $H(0)=\{e\}$.
\end{proof}

\section{Proofs of the main results}\label{sec:main-proofs}

\begin{proof}[Proof of Theorem~\ref{thm:main}]
We apply Corollary~\ref{cor:transfer} with $H=G$ and $\lalg=\g$.
Every element of $G$ can be
joined to $e$ by a morphism from a smooth irreducible curve, by
curve-connectedness and normalization. Every such curve is tangent
to $\g$, so $H(\g)=G$. Corollary~\ref{cor:transfer} gives the required
group bounds. For the reverse bounds, the proof of
\cite[Lemma~5.1.4]{KZ} applies without affineness: it uses regularity of
$\Ad$ and the identity $d\Ad_e=\ad$, supplied by
Lemmas~\ref{lem:adjoint-regular} and~\ref{lem:adjoint-differential}.
Indeed, for closed ind-subgroups $A,B\subseteq G$, differentiating
$aba^{-1}b^{-1}$ first in $b$ and then in $a$ gives
\[
 [\Lie A,\Lie B]\subseteq\Lie\overline{[A,B]}.
\]
Taking closures does not increase derived length or nilpotency class,
since the corresponding commutator identities are closed conditions.
Induction along the closed derived or lower central series gives the
reverse bounds.
\end{proof}

\begin{proposition}\label{prop:representation}
We retain the notation of Section~\ref{sec:commutator} and assume $H(\lalg)=H$.
Let $\rho:G\to\GL(V)$ be a regular linear representation on a
countable-dimensional vector space and set $\h:=d\rho_e(\lalg)$.
If $\h$ is solvable (respectively, nilpotent), then so is $\rho(H)$, with
$\dl(\rho(H))\leq\dl(\h)$ (respectively, $c(\rho(H))\leq c(\h)$).
\end{proposition}

\begin{proof}
Let $f:=d\rho_e|_{\lalg}$ and $\mathfrak n:=\ker f$. Equivariance of
$d\rho_e$ makes the inverse images under $f$ of the characteristic
ideals of $\h$ into $\Ad(H)$-stable ideals of $\lalg$.
We set $\mathfrak a_r:=f^{-1}(\h^{(r)})$. Since
\[
 \mathfrak a_0=\lalg,\qquad
 [\mathfrak a_r,\mathfrak a_r]\subseteq\mathfrak a_{r+1},
\]
Lemma~\ref{lem:commutator} and monotonicity of the construction
$\mathfrak a\mapsto H(\mathfrak a)$ give
$H^{(r)}\subseteq H(\mathfrak a_r)$.
If $\h^{(d)}=0$, then $\mathfrak a_d=\mathfrak n$.

We claim that $H(\mathfrak n)\subseteq\ker\rho$.
Let $\alpha$ be one of the curves generating $H(\mathfrak n)$.
For $v\in V$, the chain rule gives
\[
 d\bigl(\rho(\alpha)v\bigr)
 =d\rho_e(\delta^r\alpha)\,\rho(\alpha)v=0.
\]
Here $d\rho_e$ is applied to the values of the $1$-form.
By Lemma~\ref{lem:constant}, the orbit curve is constant, with value $v$
at the marked point.
Thus every generating curve of $H(\mathfrak n)$ acts trivially, proving
the claim and the derived-series assertion.

For nilpotence, we set $\mathfrak b_r:=f^{-1}(\h_r)$.
Then $\mathfrak b_1=\lalg$ and
$[\lalg,\mathfrak b_r]\subseteq\mathfrak b_{r+1}$, so the same lemma
gives $H_r\subseteq H(\mathfrak b_r)$.
When $\h_{c+1}=0$, the last subgroup is $H(\mathfrak n)$,
which acts trivially.
\end{proof}

\begin{proof}[Proof of Corollary~\ref{cor:linear}]
We apply Proposition~\ref{prop:representation} with
$H=G$ and $\lalg=\Lie G$. Connectedness gives
$H(\lalg)=G$, and the proposition gives both assertions.
\end{proof}

\begin{remark}\label{rem:linear}
Let $\rho:G\to\GL(V)$ be a regular representation as in
Corollary~\ref{cor:linear}, and set $\h:=d\rho_e(\Lie G)$.
The image $\rho(G)$ is considered as an abstract
group; no ind-group structure on $\GL(V)$ is assumed. Even when the
representation is realized in an ambient ind-group whose Lie algebra
embeds into $\End(V)$, applying Theorem~\ref{thm:main} to the connected closure of
$\rho(G)$ requires bounds for $\Lie\overline{\rho(G)}$,
whereas our hypothesis concerns only $\h$.
This leads to the following question: for a homomorphism of ind-groups
$q:G\to\mathcal K$ with $G$ connected, do we have
\[
 \Lie\overline{q(G)}=dq_e(\Lie G)?
\]
Here the closure is taken in $\mathcal K$. The inclusion from right to
left is automatic, but connectedness of the closure does not establish
the reverse inclusion. Already when $q$ is surjective, this asks whether
its differential is surjective, a question left open in
\cite[Section~1.11]{KZ}.

Our separate proof of Corollary~\ref{cor:linear} works directly with the
action on $V$ and does not require such an identification.
\end{remark}

\section{Algebraically generated groups and applications}
\label{sec:algebraically-generated}

Here, ``algebraically generated group'' is used in the sense of
\cite{KZ}, rather than in the sense of \cite{CKRVS}.
We use the notation of Corollary~\ref{cor:algebraic}.
For $m\geq1$ and $i_1,\ldots,i_m\in I$, we consider the multiplication maps
\begin{equation}\label{eq:mult}
 \mu_{i_1,\ldots,i_m}:H_{i_1}\times\cdots\times H_{i_m}\longrightarrow G,
 \qquad (h_1,\ldots,h_m)\longmapsto h_1\cdots h_m.
\end{equation}
The group $H$ is the union of their images. These images are connected
and contain $e$, so $H$ is connected.

\begin{proof}[Proof of Corollary~\ref{cor:algebraic}]
For every $i$, the subalgebra $\h$ is stable under
$\ad(\Lie H_i)$. Since $\Ad$ is regular and $H_i$ is connected,
$\h$ is $\Ad(H_i)$-stable by \cite[Corollary~7.2.7]{FK}.
Thus $\h$ is $\Ad(H)$-stable.
Every element of $H_i$ is joined to $e$ by
a smooth curve in $H_i$. The translated tangent vectors of that curve
lie in $\Lie H_i\subseteq\h$. Thus $H_i\subseteq H(\h)$
for every $i$, and $H(\h)=H$. Corollary~\ref{cor:transfer} gives
$\dl(H)\leq\dl(\h)$ when $\h$ is solvable, and
$c(H)\leq c(\h)$ when $\h$ is nilpotent.

For the converse and the reverse bounds, we set $K:=\overline H\subseteq G$.
Since $H$ is connected, so is $K$. Since $H_i\subseteq K$ for every $i$,
we have $\h\subseteq\Lie K$.
Group commutator identities pass to the closure. Thus, if $H$ is
solvable, then $K$ is solvable with $\dl(K)=\dl(H)$.
Theorem~\ref{thm:main} now gives
\[
 \dl(\h)\leq\dl(\Lie K)=\dl(K)=\dl(H).
\]
Likewise, if $H$ is nilpotent, then $K$ is nilpotent with $c(K)=c(H)$,
and the same theorem gives
\[
 c(\h)\leq c(\Lie K)=c(K)=c(H).
\]
Together with the preceding bounds, these prove both equivalences and
equalities.
\end{proof}

\begin{proof}[Proof of Corollary~\ref{cor:application-locally-finite}]
We first suppose that $\h_i=\Lie H_i$ for connected algebraic subgroups
$H_i\subseteq\Aut(X)$, and set $H:=\langle H_1,\ldots,H_n\rangle$.
The algebraic family $\mu_{1,\ldots,n}$ in \eqref{eq:mult}
has an irreducible source $H_{i_1}\times\cdots\times H_{i_m}$;
its image contains $\operatorname{id}_X$ and generates $H$.
Thus $H$ satisfies the hypothesis of Theorem~A in \cite{CKRVS}.
Corollary~\ref{cor:algebraic} and that theorem give
\begin{equation}\label{eq:algebraicity}
 \h\text{ solvable}
 \quad\Longrightarrow\quad
 H\text{ solvable}
 \quad\Longrightarrow\quad
 H\text{ is an affine algebraic group}.
\end{equation}
Now \cite[Theorem~2.5.1]{KZ} yields $\h=\Lie H$.
In particular, $\h$ is finite dimensional and acts locally finitely on
$\mathcal O(X)$.

Let us now reduce the general case to the previous one using
\cite[Corollary~2.13]{ZLF}.
By that result, $\h$ is contained in a solvable Lie algebra
$\mathfrak h_{\max}$ generated by Lie algebras of algebraic groups.
Each $\h_i$ is finite dimensional. Indeed, we choose a finite-dimensional
$\h_i$-invariant subspace of $\mathcal O(X)$ containing a finite set of
algebra generators. Restriction to that subspace is injective, since a
derivation vanishing on the generators is zero.

We choose bases for the $\h_i$. Each basis vector is a finite linear
combination of iterated brackets of elements of the Lie algebras
generating $\mathfrak h_{\max}$, so only finitely many of these Lie
algebras occur. Since there are finitely many basis vectors in total,
we can choose finitely many connected
algebraic subgroups $K_1,\ldots,K_m\subseteq\Aut(X)$ such that
\[
 \h\subseteq
 \mathfrak l:=\langle\Lie K_1,\ldots,\Lie K_m\rangle_{\mathrm{Lie}}
 \subseteq\mathfrak h_{\max}.
\]
The Lie algebra $\mathfrak l$ is solvable. By the first part of the proof,
it is locally finite and finite dimensional, and the same holds for its
subalgebra $\h$.
\end{proof}

\end{document}